\documentclass{article}
\usepackage{amssymb,amsmath,amsthm,graphicx,color}
\usepackage[all,color]{xy}
\usepackage{floatflt}
\usepackage{float}

\def\qed{\hfill {\hbox{${\vcenter{\vbox{               
   \hrule height 0.4pt\hbox{\vrule width 0.4pt height 6pt
   \kern5pt\vrule width 0.4pt}\hrule height 0.4pt}}}$}}}

\def\utr{\, \underline{\triangleright}\, }
\def\otr{\, \overline{\triangleright}\, }

\newtheorem{theorem}{Theorem}

\newtheorem{corollary}[theorem]{Corollary}

\theoremstyle{definition}
\newtheorem{example}{Example}
\newtheorem{definition}{Definition}
\newtheorem{remark}{Remark}

\date{}

\title{\Large \textbf{Biquandle Coloring Invariants of Knotoids}}

\author{Neslihan G\"ug\"umc\"u\footnote{Email:nesli@central.ntua.gr }
 \and
Sam Nelson\footnote{Email: Sam.Nelson@cmc.edu. Partially supported by Simons Foundation collaboration grant 316709}}

\begin{document}
\maketitle

\begin{abstract}

In this paper, we consider biquandle colorings for knotoids in $\mathbb{R}^2$ 
or $S^2$ and we construct several coloring invariants for knotoids derived 
as enhancements of the biquandle counting invariant. We first enhance the 
biquandle counting invariant by using a matrix constructed by utilizing the 
orientation a knotoid diagram is endowed with. We generalize Niebrzydowski's 
biquandle longitude invariant for virtual long knots to obtain new 
invariants for knotoids. We show that biquandle invariants can detect mirror 
images of knotoids and show that our enhancements are proper in the sense 
that knotoids which are not distinguished by the counting invariant are 
distinguished by our enhancements.

\end{abstract}

\parbox{5.5in} {\textsc{Keywords:} Knotoids, biquandles, enhancements of 
counting invariants

\smallskip

}

\section{\large\textbf{Introduction}}\label{I}

\textit{Knotoids}, introduced by Turaev in \cite{Tu}, are knotted curves immersed in surfaces 
with two endpoints which are not allowed to move over or under another strand, 
generalizing the notion of $(1,1)$-tangle to allow endpoints in different 
regions of the planar complement of the knotoid. Knotoids and their invariants
are the subject of much recent study \cite{GK1,GK2,GL2,GL1,Gthesis}.

\textit{Biquandles} are algebraic structures with axioms encoding the 
oriented Reidemeister moves; they are used to define invariants
of oriented knots and links by counting colorings. Introduced in \cite{FRS},
they have also been the subject of much recent study \cite{EN}. In particular,
biquandle-colored knots and links form a category whose invariants can be used
to define invariants of knots and links known as 
\textit{enhancements} of the counting invariant \cite{EN,CJKLS,NOR,NO}.

\textit{Longitudes} in the knot group are elements of the knot group 
which are homotopic to the knot itself; these have been studied in various
contexts such as \cite{K}. In \cite{NL} a type of enhancement of the biquandle 
counting invariant of virtual long knots was defined using a biquandle version 
of longitudes, known as \textit{biquandle longitude invariants}.

In this paper we consider biquandle colorings of knotoids and introduce several enhancements of the biquandle counting invariant. First, we introduce a
matrix-valued enhancement of the counting invariant making use of the special
structure of knotoids. Next, we apply the biquandle longitude enhancement 
idea to the case of knotoids, obtaining several enhanced invariants in the 
process. The paper is organized as follows. In Section \ref{K} we recall the 
basics of knotoids. In Section  \ref{B} we review biquandles and the biquandle 
counting invariant, which can be defined for knotoids without any further 
modification, and we use the fact that
biquandle colorings of knotoids are fixed at the endpoints to enhance the
counting invariant, obtaining a matrix-valued invariant. In Section \ref{SB} 
we define several longitude enhancements for knotoids and give some examples. 
We show that biquandle invariants can detect mirror images of knotoids and 
show that our enhancements are proper in the sense that knotoids which are not
distinguished by the counting invariant are distinguished by our enhancements.
We conclude in Section \ref{Q} with some questions for future work.

\section{\large{\textbf{Knotoids}}}\label{K}

Knotoid diagrams are open ended oriented knot diagrams in oriented surfaces, 
forming new diagrammatic theories, including an extension of the classical 
knot theory when the supporting surface of the knotoid is considered to be 
the $2$-sphere $S^2$ \cite{Tu}. A \textit{knotoid diagram} $K$ in a surface 
$\Sigma$ is a generic immersion of the unit interval $[0,1]$ into $\Sigma$ 
with two distinct endpoints as images of $0$ and $1$, named as the 
\textit{tail} and \textit{head} of $K$, a finite number of self-crossing 
points that are transverse double points endowed with  extra passage 
information of over/under, and endowed with the orientation of $[0,1]$ from 
left to right (see Figure 1, for examples). A $(1,1)$-tangle or a 
long knot is defined to have its endpoints in a single region of the 
diagram. A knotoid diagram generalizes the notion of $(1,1)$- tangle or long 
knot by allowing its two endpoints to lie in different regions of the diagram. 
Below are some examples of knotoid diagrams considered in $S^2$ or $\mathbb{R}^2$:
\[\includegraphics{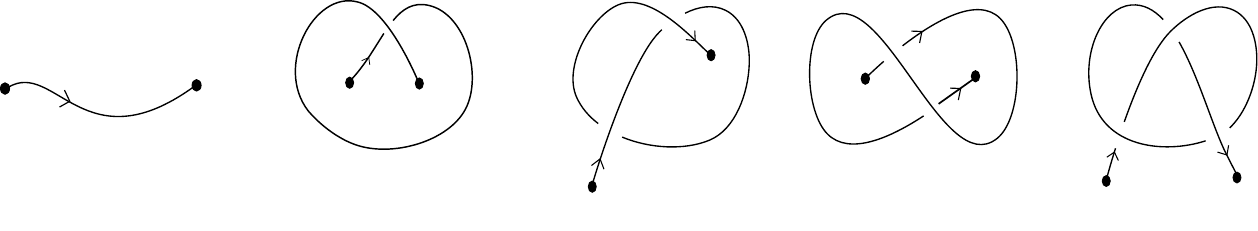}\] 

A \textit{knotoid} in a surface is defined to be an 
equivalence class given by the isotopy relation generated by three 
Reidemeister moves each taking place in local disks disjoint from the 
endpoints of the knotoid diagram, and isotopy of 
the supporting surface \cite{Tu}. The \textit{trivial knotoid} is a knotoid 
admitting a diagram without any crossings in its equivalence class. We forbid the following two moves 
\[\scalebox{0.4}{\includegraphics{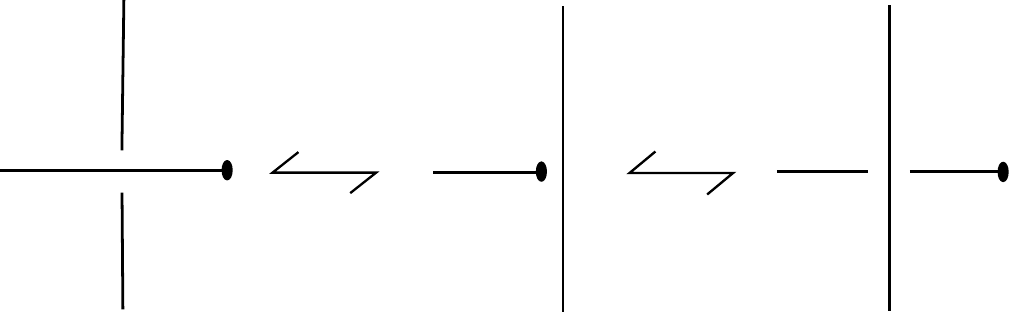}}\]
displacing an endpoint by sliding it over/under a 
transversal strand. The ban of these moves yields a non-trivial theory of knotoids. 

The theory of knotoids was introduced by Vladimir Turaev \cite{Tu} in 2012 and 
further studied by the first author and Louis Kauffman \cite{GK1} with the introduction of new invariants for knotoids. We also 
have a classification table of spherical knotoids up to 5 crossings given by 
Andrew Bartholomew \cite{Ba} given by the use of a generalization of the bracket polynomial for knotoids introduced by Turaev \cite{Tu}. Furthermore, geometric interpretations of 
knotoids (both planar and spherical) given in \cite{GK1, GK2} make it an eligible 
theory for understanding physical structures, for instance see \cite{GDFS, GGLDSK} for 
topological modeling of open linear protein chains via spherical and planar knotoids. Recently, a braided counterpart theory to the theory of planar knotoids, namely the theory of braidoids, has been introduced by the first author and Sofia Lambropoulou \cite{GL1}. In \cite{GL2, GL1}, they proved analogues of the Alexander theorem and the Markov theorem for knotoids/multi-knotoids and braidoids.


\section{\large\textbf{Biquandles and the Counting Invariant}}\label{B}

In this section, we review biquandles and the biquandle counting invariant. 
There is a variety for the choice of notation in defining a biquandle, we 
will find the notation in \cite{EN2} most suitable for the purposes of this
paper.

\begin{definition}
Let $X$ be a set. A \textit{biquandle structure} on $X$ is an assignment of
two bijections $\alpha_b,\beta_b:X\to X$ to each element $b\in X$ such that 
the following axioms are satisfied:
\begin{itemize}
\item[(i)] For all $a\in X$, $\alpha_a(a)=\beta_a(a)$,
\item[(ii)] For all $a,b\in X$, the map $S:X\times X\to X$ defined by
\[S(a,b)=(\alpha_a(b),\beta_b(a))\]
is invertible, and
\item[(iii)] For all $a,b\in X$ we have the \textit{exchange laws}:
\[\begin{array}{rcll}
\alpha_{\alpha_a(b)}\alpha_a & = &\alpha_{\beta_b(a)}\alpha_b & \mathrm{(iii.i)} \\
\beta_{\alpha_a(b)}\alpha_a & = &\alpha_{\beta_b(a)}\beta_b & \mathrm{(iii.ii)} \\
\beta_{\beta_a(b)}\beta_a & = &\beta_{\alpha_b(a)}\beta_b & \mathrm{(iii.iii)} \\
\end{array}\]
\end{itemize}
\end{definition}

\begin{example}
Examples of biquandles include 
\begin{itemize}
\item \textit{Constant Action Biquandles}: For any set $X$ and bijection
$\sigma:X\to X$, the assignment $\alpha_b=\beta_b=\sigma$ define a biquandle
structure in which the actions $\alpha_b,\beta_b$ do not depend on $b$.
\item \textit{Alexander Biquandles}: For any module over the 2-variable
Laurent polynomial ring $\mathbb{Z}[t^{\pm 1}, s^{\pm 1}]$, the maps
\[\alpha_b(a) = sa \quad\mathrm{and}\quad \beta_b(a)=ta+(s-t) b\]
define a biquandle structure.
\item \textit{Quandles}: A biquandle in which $\alpha_b$ is the identity map
for all $b$, is a \textit{quandle}. Examples include groups with 
$\beta_b(a)=b^{-n}ab^n$ for $n\in\mathbb{Z}$ (known as \textit{$n$-fold 
conjugation quandles}), groups with $\beta_b(a)=ab^{-1}a$ (known as \textit{core 
quandles}), and vector spaces with 
$\beta_b(a)=a+\langle a,b\rangle b$ for a symplectic form $\langle .,. \rangle$
(known as \textit{symplectic quandles}) as well as many more.
\end{itemize}
See \cite{EN} for more.
\end{example}

\begin{example}
We can represent biquandle structures on the finite set $X=\{1,2,\dots, n\}$
by listing the maps $\beta_1,\dots, \beta_n,\alpha_1,\dots, \alpha_n$
as column vectors in an $n\times 2n$ block matrix with the first through
$n$th columns being $\beta_1$ through $\beta_n$ and the $n+1$st though
$2n$th columns being $\alpha_1$ through $\alpha_n$, as any map $\sigma:X\to X$
can be represented by the column vector
\[\left[\begin{array}{c}
\sigma(1) \\\sigma(2) \\ \vdots \\\sigma(n)
\end{array}\right].\]
For instance, the Alexander biquandle structure on $\mathbb{Z}_4$ with
$t=1$ and $s=3$ has matrix
\[\left[\begin{array}{rrrr|rrrr}
3 & 1 & 3 & 1 & 3 & 3 & 3 & 3\\
4 & 2 & 4 & 2 & 2 & 2 & 2 & 2\\
1 & 3 & 1 & 3 & 1 & 1 & 1 & 1 \\
2 & 4 & 2 & 4 & 4 & 4 & 4 & 4
\end{array}\right].\]
Then for example in cycle notation we have $\beta_1=(13)(24)$ and 
$\alpha_1=(13)$ while $\beta_2=()$ where () stands for the identity permutation, and $\alpha_2=(13)$.
The notation is chosen so that the two block matrices are the operation tables
for the binary operations $\utr,\otr$ on $X$ defined by $x\utr y=\beta_y(x)$
and $x\otr y=\alpha_y(x)$.
\end{example}

\begin{definition}\label{def:1}
Let $K$ be a knotoid in $S^2$ or $\mathbb{R}^2$ represented by a diagram $D$ and $X$ a biquandle. A \textit{semiarc} of $D$ is defined to be a piece of strand of $D$ that connects either two adjacent crossings or an endpoint to a crossing.  
A \textit{biquandle coloring} of $D$ by $X$ is an assignment of an element
of $X$ to each semiarc in $D$ such that at every crossing the local coloring is as illustrated in the figure
below. Here we label (or color) the four semiarcs neighboring a positive and a negative crossing, respectively and all directed downwards. The left-hand side colors $a, b \in X$ are considered to be the inputs for the corresponding biquandle bijections on $X$, $\beta_b$ and $\alpha_a$, respectively. The upper right semiarc is colored with the output color $\alpha_a(b)$ and the lower right semiarc is colored $\beta_b(a)$, if the crossing is positive. If the crossing is negative then the upper right semiarc is colored $\beta_b(a)$ and the lower right semiarc is colored $\alpha_a(b)$. 
\[\includegraphics{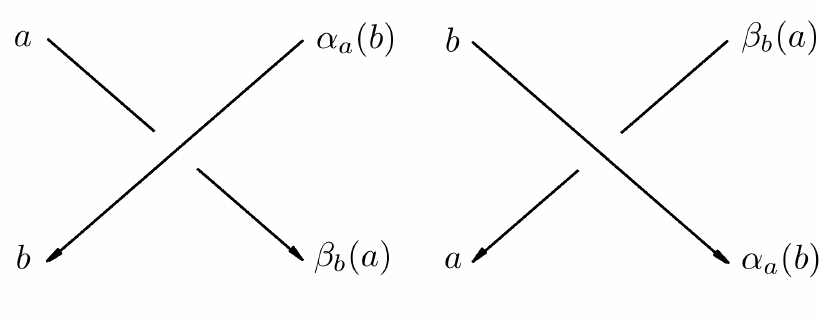}\]
\end{definition}

Similarly to the case of (classical/virtual) knots, the biquandle axioms are set so that for any diagram with a biquandle
coloring before a Reidemeister move, there is a unique coloring of the 
diagram after the move which agrees with the previous coloring outside
the neighborhood of the move. It follows that the number of such colorings
is an invariant of knotoids. We will denote the set of biquandle colorings
of a knotoid diagram $D$ by a biquandle $X$ as $\mathcal{C}(D,X)$. 

\begin{definition}
Let $X$ be a finite biquandle and $D$ a knotoid diagram representing a 
knotoid $K$. Then the \textit{biquandle counting invariant} of $K$ is
the cardinality $\Phi_X^{\mathbb{Z}}(K)=|\mathcal{C}(D,X)|$ of the set of
colorings of $D$ by $X$.
\end{definition}

\begin{example}\label{ex:kc}
We can compute $\Phi^{\mathbb{Z}}_X(K)$ from any diagram of $K$ by listing all
possible assignments of elements of $X$ to the semiarcs in $K$ and noticing
which ones satisfy the crossing relations pictured in Definition \ref{def:1}.
For example, the knotoid diagram below has five semiarcs and four crossing
equations; arbitrarily selecting $a$ and $b$ in $X$ determines $c=\alpha_a(b)$, $d=\beta_b(a)$ 
and $e=\beta_c(b)=\beta_{\alpha_a(b)}(b)$, so the coloring is valid provided
we have $d=\beta_b(a)=\alpha_b(c)=\alpha_b(\alpha_a(b))$. Then for instance if
$X= \{1,2,3,4,5\}$ is equipped with the biquandle structure determined with the listed matrix below, we can check the required equation
for each assignment of elements of $X$ to $a,b$.
\[\raisebox{-1.5in}{\includegraphics{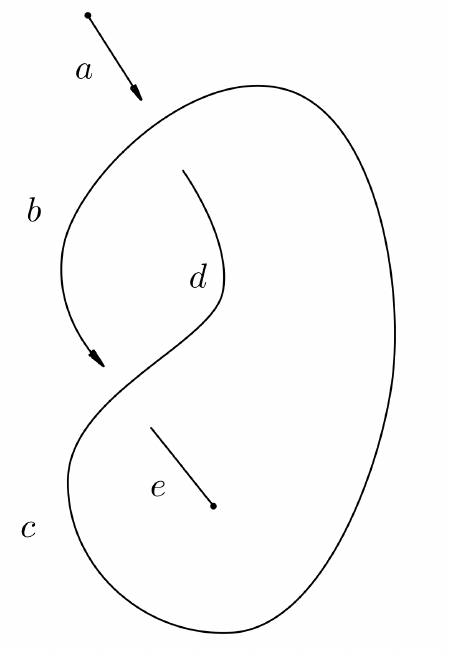}} \quad 
\begin{array}{c}
X=\left[\begin{array}{rrrrr|rrrrr}
3& 1& 3& 1& 1 & 3& 3& 3& 3& 3\\
5& 4& 5& 2& 5 & 5& 4& 5& 4& 4\\
1& 3& 1& 3& 3 & 1& 1& 1& 1& 1\\
2& 2& 2& 5& 4 & 2& 5& 2& 5& 5\\
4& 5& 4& 4& 2 & 4& 2& 4& 2& 2
\end{array}\right] \\
 \ \\
\begin{array}{rr|c||rr|c} 
a & b & \beta_b(a) = \alpha_b(\alpha_a(b))? &
a & b & \beta_b(a) = \alpha_b(\alpha_a(b))? \\\hline
1 & 1 &  -  & 3 & 4 & - \\
1 & 2 &  -  & 3 & 5 & - \\
1 & 3 & \checkmark   & 4 & 1 & - \\
1 & 4 &  -  & 4 & 2 & - \\
1 & 5 &  -  & 4 & 3 & - \\
2 & 1 &  -  & 4 & 4 & - \\
2 & 2 &  -  & 4 & 5 & \checkmark \\
2 & 3 &  -  & 5 & 1 & - \\
2 & 4 & \checkmark   & 5 & 2 & \checkmark \\
2 & 5 &  -  & 5 & 3 & - \\
3 & 1 & \checkmark   & 5 & 4 & - \\
3 & 2 &  -  & 5 & 5 & - \\
3 & 3 & - & & & \\
\end{array}
\end{array}\]
Hence in this example we see that $\Phi_X^{\mathbb{Z}}(K)=5$.
\end{example}

\begin{example}\label{ex:mr}
Let $X$ be the biquandle with matrix
\[\left[\begin{array}{rrr|rrr}
2 & 1 & 3 & 2 & 2 & 2\\
1 & 3 & 2 & 3 & 3 & 3\\
3 & 2 & 1 & 1 & 1 & 1
\end{array}\right].\]
As can be verified by the table below, the knotoid in Example \ref{ex:kc} has three colorings by $X$. We see that for coloring the
mirror image knotoid obtained by over-under switching
all of the crossings of $D$, the equation $d= \alpha_b(a)=\beta_b(c)=\beta_b(\beta_a(b))$ should be satisfied. It is verified in the same table that the mirror image knotoid has no colorings by $X$; hence
the biquandle counting invariant shows that this knotoid is not equivalent 
to its mirror image and that biquandle counting invariants can detect
this type of mirror image in knotoids.
\[\raisebox{-1.5in}{\includegraphics{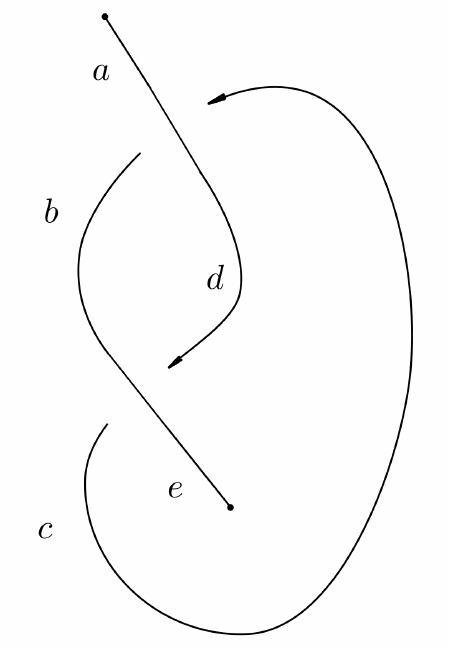}} \quad 
\begin{array}{c}
\begin{array}{rr|c|c} 
a & b & \alpha_b(a) = \beta_b(\beta_a(b))? &
 \beta_b(a) = \alpha_b(\alpha_a(b))? \\ \hline
1 & 1 &  2\ne 1  & 2\ne 3\\
1 & 2 &  2\ne 1  & 1= 1 \checkmark\\
1 & 3 &  2\ne 1  & 3\ne 2\\
2 & 1 &  3\ne 2  & 1\ne 3 \\
2 & 2 &  3\ne 2  & 3\ne 1 \\
2 & 3 &  3\ne 2  & 2=2 \ \checkmark\\
3 & 1 &  1\ne 3  & 3=3 \ \checkmark\\
3 & 2 &  1\ne 3  & 2\ne 1 \\
3 & 3 &  1\ne 3  & 1\ne 2 \\
\end{array}
\end{array}\]
\end{example}

\begin{example}\label{ex:al}
For Alexander biquandles of the form $X=\mathbb{Z}_n$ with $s,t$ coprime 
to $n$, we can find the set of colorings via linear algebra 
since the coloring equations 
\[\alpha_a(b)=sb\quad \mathrm{and}\quad \beta_b(a)=ta+(s-t)b\]
are linear. The knotoid $K$ in Example
\ref{ex:kc} has system of coloring equations
\[\begin{array}{rcl}
\beta_b(a) & = & d \\
\alpha_a(b) & = & c\\
\beta_c(b) & = & e \\
\alpha_b(c) & = & d
\end{array}\]
which for an Alexander biquandle becomes
\[\begin{array}{rcl}
\begin{array}{rcl}
ta+(s-t)b & = & d \\
sb & = & c\\
tb+(s-t)c & = & e \\
sc & = & d.
\end{array}
\end{array}
\]
Then for example to find colorings of $K$ by the Alexander biquandle 
$X=\mathbb{Z}_5$ with $t=2$ and $s=3$, we can find the kernel of the 
homogeneous system's coefficient matrix over $\mathbb{Z}_5$:
\[\left[\begin{array}{rrrrr}
2 & 1 & 0 & 4 & 0 \\
0 & 3 & 4 & 0 & 0 \\
0 & 2 & 1 & 0 & 4 \\
0 & 0 & 3 & 4 & 0
\end{array}\right]
\leftrightarrow
\left[\begin{array}{rrrrr}
1 & 0 & 0 & 4 & 0 \\
0 & 1 & 0 & 1 & 0 \\
0 & 0 & 1 & 3 & 0 \\
0 & 0 & 0 & 0 & 1 \\
\end{array}\right]
\] 
so in this case the kernel is one-dimensional 
spanned by $(1,4,2,1,0)$ and there are $|X|=5$
colorings.
\end{example}
\section{Enhancements of the Counting Invariant}

\subsection{Biquandle counting matrix}

Unlike the case of classical and virtual knots, a biquandle-colored knotoid 
has a well-defined initial color and terminal color assigned to the semiarcs 
adjacent to endpoints. We make use of this feature to enhance the 
counting invariant as follows.

\begin{definition}
Let $X=\{x_1,\dots, x_n\}$ be a finite biquandle and let $\mathcal{C}_{jk}(K)$
be the set of $X$-colorings of $K$ in which the initial semiarc is colored 
$x_j$ and the final semiarc is colored $x_k$. Then the \textit{biquandle 
counting matrix} of $K$ with respect to $X$ is the matrix $\Phi_X^{M_n}(K)$ 
whose entry in row $j$ column $k$ is $|\mathcal{C}_{jk}(K)|$.
\end{definition}

\begin{theorem}
For any finite biquandle $X$, $\Phi_X^{M_n}(K)$ is an invariant of knotoids.
\end{theorem}

\begin{proof}
This follows immediately from the facts Reidemeister moves take place away 
from the endpoints so that the biquandle colorings of the initial and the 
terminal semiarcs and also that biquandle coloring numbers are
unchanged by Reidemeister moves.
\end{proof}

\begin{corollary}
For any knotoid $K$ and finite biquandle $X$, we have
\[\Phi_X^{\mathbb{Z}}(K)=\sum_{1\le j,k\le n} (\Phi_X^{M_n}(K))_{j,k}.\]
That is, the biquandle counting invariant is the sum of the entries of the 
biquandle counting matrix.
\end{corollary}

\begin{example}
We computed the biquandle coloring matrix for each of the knotoids in the 
table at \cite{Ba} with respect to the biquandle in Example \ref{ex:mr}. 
The results are listed in the table.
\[\begin{array}{r|l}
\Phi_X^{M_n}(K) & K \\ \hline 
& \\
\left[\begin{array}{rrr}
0 & 0 & 0 \\
0 & 0 & 0 \\
0 & 0 & 0 \\
\end{array}\right]
& 2.1, 4.4, 4.5, 5.1, 5.2, 5.3, 5.4, 5.10, 5.11, 5.12, 5.13, 5.14, 5.15, 5.16,
5.18, 5.21, 5.23 \\
& \\
\left[\begin{array}{rrr}
0 & 0 & 1 \\
1 & 0 & 0 \\
0 & 1 & 0 \\
\end{array}\right]
& 3.1, 4.3, 4.6, 4.7, 5.7, 5.27, 5.29 \\
& \\
\left[\begin{array}{rrr}
0 & 1 & 0 \\
0 & 0 & 1 \\
1 & 0 & 0 \\
\end{array}\right]
& 4.1, 4.2, 4.8, 5.8, 5.19, 5.20, 5.26, 5.28 \\
& \\
\left[\begin{array}{rrr}
1 & 0 & 0 \\
0 & 1 & 0 \\
0 & 0 & 1 \\
\end{array}\right]
& 4.9, 5.5, 5.6, 5.9, 5.22, 5.24, 5.25 \\
& \\
\left[\begin{array}{rrr}
3 & 0 & 0 \\
0 & 3 & 0 \\
0 & 0 & 3 \\
\end{array}\right] &
5.17 \\
& \\
\left[\begin{array}{rrr}
0 & 3 & 0 \\
0 & 0 & 3 \\
3 & 0 & 0 \\
\end{array}\right] &
5.30 \\
\end{array}\]
In this example there are three counting invariant values, namely 
$\Phi_X^{\mathbb{Z}}(K)=0,3$ and $9$, while we have six counting matrix values.
\end{example}

\subsection{\large\textbf{Longitude enhancements}}\label{SB}

In this section, we recall the biquandle longitude invariant and consider 
its application to the case of knotoids. We begin with a definition adapted
from \cite{NL}.

\begin{definition}
Let $X$ be a finite biquandle and $D_f$ a knotoid diagram with a coloring $f$
by $X$. Traveling along the knotoid in the direction of the orientation, we
encounter each crossing point twice. We will assign a bijection 
$\beta^{\pm 1}_{L}:X\to X$ to each of the two passes through each crossing, depending on the crossing sign, the type of the pass (over or under) and biquandle 
coloring as shown:
\[\scalebox{0.9}{\includegraphics{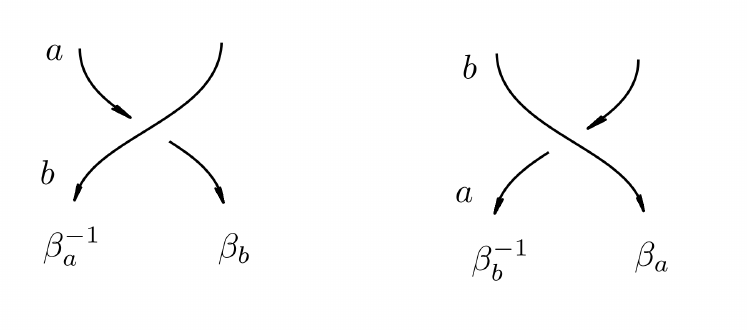}}\]
We can summarize the rules for defining this bijection as follows: When traveling through a crossing,
the assigned bijection is $\beta^{jk}_{L}$ where $j$ is the crossing sign, $k$ is $1$
if we are going under and $-1$ if we are going over, and $L$ is the biquandle 
color of the strand we see if we look to the right as we pass through the 
crossing.

Then the \textit{biquandle longitude weight} of the $X$-colored diagram
$D$, denoted by $BLW(D_f)$, is the composition of the crossing bijections
in the order in which they are encountered when traveling along the knotoid.
\end{definition}

\begin{remark}
Our convention is different from that of \cite{NL}, where the author chooses 
the color of the strand on the left. Our choice is motivated by our choice 
of biquandle notation, which also differs from the convention of \cite{NL}. 
\end{remark}

As in \cite{NL}, we have the following theorem.
\begin{theorem}
$BLW(D_f)$ is not changed by biquandle-colored Reidemeister moves.
\end{theorem}

\begin{proof}
It is sufficient to check on a generating set of biquandle-colored 
oriented Reidemeister moves; one such set includes all four oriented 
moves of type I and of type II together with the Reidemeister III move
with all positive crossings. In the cases of moves of type I and II, our 
choice of coloring rules yields an identity map on each strand.
\[\includegraphics{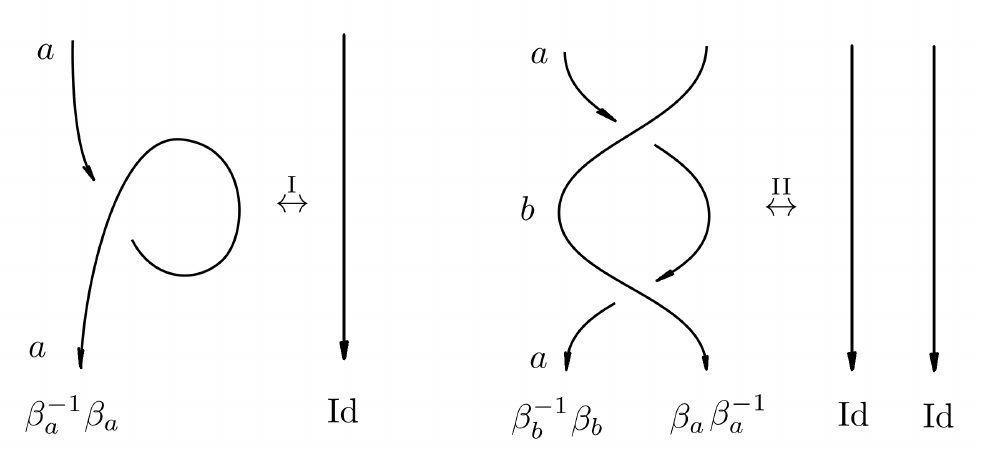}\]
We depict one case of type I and one of type II; the others are similar.
In the case of type III, the maps induced on each of the three strands
are equivalent by biquandle exchange law (iii.iii).
\[\includegraphics{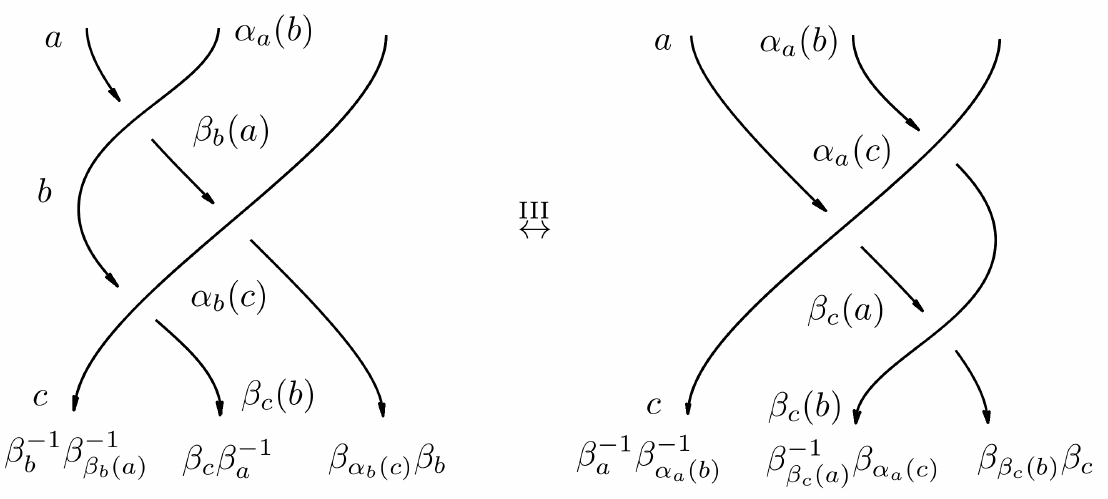}\]
\end{proof}

\begin{example}\label{ex:1}
Let us consider the $X$-coloring $D_f$ of the knotoid diagram
from Example \ref{ex:al} depicted below.
\[\includegraphics{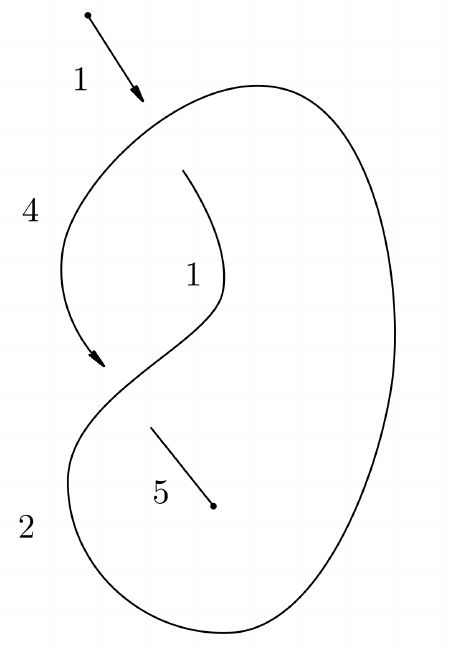}\]
The matrix of $X$ is given by
\[\left[\begin{array}{rrrrr|rrrrr}
3 &4 &5 &1 &2 & 3 &3 &3 &3 &3 \\
5 &1 &2 &3 &4 & 1 &1 &1 &1 &1 \\
2 &3 &4 &5 &1 & 4 &4 &4 &4 &4 \\
4 &5 &1 &2 &3 & 2 &2 &2 &2 &2 \\
1 &2 &3 &4 &5 & 5 &5 &5 &5 &5
\end{array}\right]\]
where our columns are numbered $1$ though $5$ in each block, with $5$ 
representing the class of $0\in\mathbb{Z}_5$ so our column labels can start
with 1. We then have $\beta_1=(1325)$, $\beta_2=(1452)$, $\beta_3=(1534)$,
$\beta_4=(2354)$ and $\beta_5=(1243)$. Let us find the biquandle longitude
weight for this coloring.

Traveling along the knotoid and looking to the right, we first go under
a negative crossing with a strand labeled 4, then over a negative
crossing with a strand labeled 4, then over a negative crossing with a
strand labeled 1, then finally under a negative crossing with a
strand labeled 2, so we have (composing right to left)
\begin{eqnarray*}
BLW(D_f) 
& = & \beta^{(1)(-1)}_2\beta^{(-1)(-1)}_1\beta_4^{(-1)(-1)}\beta^{(1)(-1)}_4\\
& = & \beta^{-1}_2\beta_1\beta_4\beta^{-1}_4\\
& = & \beta_2\beta_1^{-1}\\
& = & (1452)(1523)=(12345)
\end{eqnarray*}
and this coloring has biquandle longitude weight $(12345)$.
\end{example}

\begin{definition}
Let $X$ be a biquandle and $K$ a knotoid with diagram $D$. The 
\textit{biquandle longitude multiset} invariant of $K$ is the multiset 
of biquandle longitude weights over the set of biquandle colorings of $K$. 
That is,
\[\Phi_X^{M,L}(K)=\{BLW(D_f)\ | \ f\in\mathrm{Hom}(\mathcal{B}(K),X)\}.\]
\end{definition}

For ease of comparison, we can extract polynomial invariants of knotoids from
$\Phi_X^{M,L}(K)$ by replacing each weight with some integer-valued invariant
of bijections. For instance, we can define the 
\textit{biquandle longitude exponent polynomial}  $\Phi_X^{BLE}(K)$
of $K$ by replacing each
biquandle longitude weight $BLW(D_f)$ with its exponent $exp(BLW(D_f))$, 
i.e. the minimal positive integer $k$ such that 
$BLW(D_f)^k=\mathrm{Id}$, and converting the resulting multiset to 
polynomial form
by summing the power $u^{k}$ of a formal variable $u$ for each element $k$ in 
the multiset. This polynomial notation has the advantage that evaluation
at $u=1$ yields the biquandle counting invariant $\Phi_X^{\mathbb{Z}}(K)$.

\begin{example}
The knotoid $K$ in Example \ref{ex:1} has biquandle longitude multiset
\[\Phi_X^{M,L}(K)=\{(12345),(13524),(14253),(15432),()\}\]
which yields exponent polynomial
\[\Phi_X^{exp,L}(K)=u+4u^5.\]
\end{example}

\begin{example}
We computed the biquandle longitude exponent polynomial for the knotoids
in the table at \cite{Ba} with the biquandle $X$ given by the matrix
\[
\left[\begin{array}{rrrr|rrrr}
1& 2& 1& 4& 1& 1& 1& 1 \\ 
2& 4& 4& 1& 4& 4& 4& 4 \\
3& 3& 3& 3& 3& 3& 3& 3 \\
4& 1& 2& 2& 2& 2& 2& 2 \\
\end{array}\right]
\]
using our \texttt{python} code. The results are collected in the table.
\[
\begin{array}{r|l}
\Phi_X^{exp,L}(K) & K \\ \hline
2u+2u^3 & 2.1, 3.1, 4.1, 4.2, 4.3, 4.5, 4.9,5.3, 5.4, 5.7, 5.8, 5.10, 5.12 \\
& 5.14, 5.16, 5.19, 5.20, 5.22, 5.23, 5.26, 5.27, 5.28, 5.29 \\
4u & 4.4, 4.8, 5.5, 5.6, 5.9, 5.11, 5.13, 5.15, 5.18, 5.21, 5.30 \\
4u+6u^3 & 4.6, 5.1, 5.2, 5.24, 5.25 \\
10u & 4.7, 5.17.
\end{array}
\]
In particular, the two counting invariant values $\Phi_X^{\mathbb{Z}}(K)=4,10$
are refined into four exponent polynomial values, distinguishing more knotoids
than the unenhanced counting invariant.
\end{example}

The longitude weight of a biquandle-colored knotoid is a bijection
from the biquandle to itself, since it is composition of bijections 
$\beta_L^{\pm 1}$. In the case of biquandles for which we have explicit
formulas for $\beta_L^{\pm 1}$, we can obtain explicit formulas for the biquandle
longitudes as well. This can give us a stronger enhancement than the
exponent polynomial since different longitude weights may have the same 
exponent, while still being easier to compare invariant values visually.

\begin{example}
In Example \ref{ex:al}, the biquandle $X$ has Alexander
biquandle structure on $\mathbb{Z}_5$ with
\[\beta_b(x)=2x+b\quad\mathrm{and}\quad\beta_b^{-1}(x)=3x+2b.\] 
Then the longitude for the pictured coloring is the map
\begin{eqnarray*}
BLW(D_f)(x) 
& = & \beta_2^{-1}(\beta_1^{-1}(x))\\
& = & \beta_2^{-1}(3x+2(1))\\
& = & 2(3x+1)+2\\
& = & x+4.
\end{eqnarray*}
\end{example}

This observation inspires the following definition.

\begin{definition}
Let $X$ be an Alexander biquandle and $K$ a knotoid. Then the 
\textit{Alexander Longitude Enhancement}, denoted $\Phi_X^{M,AL}(K)$ 
is the multiset of biquandle
longitude weights written explicitly as functions on $X$.
\end{definition}

\begin{example}
Let $X=\mathbb{Z}_3=\{1,2,3\}$ with $t=1$ and $s=2$. The knotoid $2.1$ below 
has the pictured Alexander biquandle colorings by $X$ yielding the listed 
longitudes.
\[\includegraphics{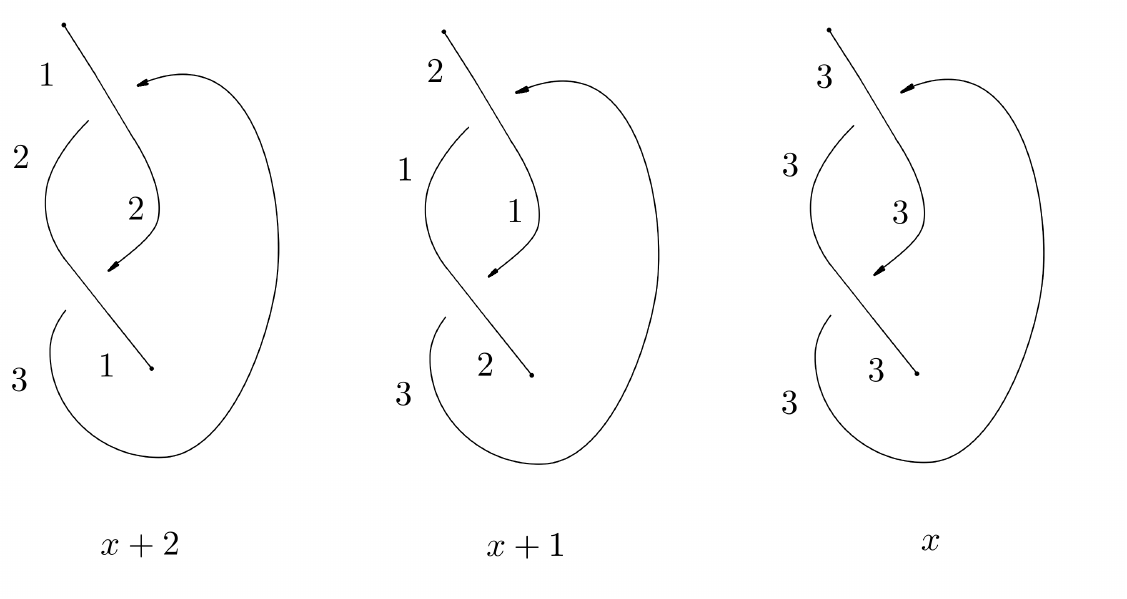}\]
Then comparing with the trivial knotoid, the counting invariant values are 
both $\Phi_X^{\mathbb{Z}}(K)=3$ while the Alexander longitude enhancement
values distinguish the knotoids, with
\[\Phi_X^{M,AL}(2.1)=\{x,x+1,x+2\}\ne\{x,x,x\}
=\Phi_X^{M,AL}(\mathrm{Trivial\ knotoid})
.\]
\end{example}

We note that one can also use the $\alpha^{-1}$ maps in place of the 
$\beta$ maps to obtain a generally distinct longitude weight, which we 
will denote by $BLW_{\alpha}(D_f)$ (and when doing so, denote 
$BLW(D_f)=BLW_{\beta}(D_f)$). 
If $X$ is a quandle, for instance, then the $\alpha_x$ 
maps are all the identity and all of the longitude information lies in the
$\beta_x$ maps. For non-quandle biquandles, the information is distributed
in the two maps. We can thus define a two-variable polynomial invariant of 
knotoids using the two weights:

\begin{definition}
Let $X$ be a biquandle and $K$ a knotoid. Then the 
\textit{biquandle longitude pair multiset} is
\[\Phi_X^{M,BLW^2}(K)
=\{(BLW_{\beta}(D_f),BLW_{\alpha}(D_f))\ |\ D_f\in\mathcal{C}(D,X)\}\]
and the \textit{biquandle longitude exponent pair polynomial} is
\[\Phi_X^{BLE^2}(K)
=\sum_{D_f\in\mathcal{C}(D,X)}u^{exp(BLW_{\beta}(D_f))}v^{exp(BLW_{\alpha}(D_f))}\] 
\end{definition}

\begin{example}
Let $X$ be the biquandle with matrix 
\[X=\left[\begin{array}{rrrr|rrrr}
2 & 2 & 2 & 2 & 2 & 3 & 1 & 4 \\
1 & 1 & 1 & 1 & 4 & 1 & 3 & 2 \\
4 & 4 & 4 & 4 & 3 & 2 & 4 & 1 \\
3 & 3 & 3 & 3 & 1 & 4 & 2 & 3
\end{array}\right].\]
Using our \texttt{python} code, we computed $\Phi_X^{BLE}(K)$ and
$\Phi_X^{BLE^2}(K)$ for the knotoids in the table at \cite{Ba}.
\[\begin{array}{r|l}
\Phi_X^{BLE}(K) & K \\ \hline
0 & 4.1, 4.2 4.6,4.7, 5.3, 5.4, 5.7, 5.8, 5.13, 5.14, 5.16, 5.20, 5.21\\
4u & 2.1, 3.1, 4.3, 4.4, 4.5, 4.8, 4.9, 5.5, 5.6, 5.9, 5.10, 5.11, 5.12, 5.15,
5.18, 5.19, 5.26, 5.27, 5.28, 5.29, 5.30 \\
16u & 5.1, 5.2, 5.17, 5.22, 5.23, 5.24, 5.25. \\
\end{array}
\]
While $\Phi_X^{BLE}(K)$  does not distinguish any knotoids
on the table which are not already distinguished by the counting invariant,
$\Phi_X^{BLE}(K)$ does:
\[\begin{array}{r|l}
\Phi_X^{BLE^2}(K) & K \\ \hline
0 & 4.1, 4.2 4.6,4.7, 5.3, 5.4, 5.7, 5.8, 5.13, 5.14, 5.16, 5.20, 5.21\\
4uv & 3.1, 4.3, 4.4, 4.5,  4.9, 5.5, 5.6, 5.9, 5.10, 5.11, 5.12, 5.15, 5.26, 5.27, 5.28, 5.29, 5.30 \\
4uv^2 & 2.1,  4.8,  5.18, 5.19 \\
16uv &  5.17, 5.24, 5.25 \\
16uv^2 & 5.1, 5.22 \\
12uv^2+4uv &  5.22, 5.23. \\
\end{array}
\]
\end{example}

Finally, we can use biquandle longitude weights to further enhance the
matrix version of the counting invariant, making use of both the longitude
information and the initial and terminal color information, analogously to the
invariants of long virtual knots described in \cite{NL}. 

\begin{definition}
Let $X$ be a biquandle and $K$ a knotoid with diagram $D$. The 
\textit{biquandle longitude exponent matrix} invariant of $K$ is the matrix 
$\Phi_X^{M_n,BWE}(K)$ whose entry in row $j$ column $k$ is 
$\sum_{D_f\in\mathcal{C}_{jk}(K)}u^{BWE(D_f)}$ and the \textit{biquandle longitude 
exponent pair matrix} invariant of $K$ is the matrix $\Phi_X^{M_n,BWE^2}(K)$
whose entry in row $j$ column $k$ is 
$\sum_{D_f\in\mathcal{C}_{jk}(K)}u^{BWE_{\beta}(D_f)}v^{BWE_{\alpha}(D_f)}$.
\end{definition}

\begin{example}
Let $X$ be the biquandle with matrix 
\[X=\left[\begin{array}{rrrr|rrrr}
3 & 1 & 2 & 4 & 3 & 3 & 3 & 3 \\
4 & 2 & 1 & 3 & 2 & 2 & 2 & 2 \\
1 & 3 & 4 & 2 & 4 & 4 & 4 & 4 \\
2 & 4 & 3 & 1 & 1 & 1 & 1 & 1 
\end{array}\right].\]
Using our \texttt{python} code, we computed $\Phi_X^{M_nBLE^2}(K)$ and
for the knotoids in the table at \cite{Ba}. The results are collected in the 
table.
\[\begin{array}{r|l}
\Phi_X^{M_n,BLE^2}(K) & K \\ \hline
& \\
\left[\begin{array}{cccc}
u^2v & 0 & 0 & 0 \\
0 & uv & 0 & 0 \\
0 & 0 & u^2v & 0 \\
0 & 0 & 0 & u^2v
\end{array}\right] & 2.1, 3.1, 4.8, 5.18, 5.26, 5.27 \\ 
& \\
\left[\begin{array}{cccc}
0 & 0 & 0 & 0 \\
0 & 3u^2v+uv & 0 & 0 \\
0 & 0 & 0 & 0 \\
0 & 0 & 0 & 0
\end{array}\right] & 4.1, 4.2, 4.6, 4.7, 5.3, 5.4, 5.7, 5.8, 5.13, 5.14, 5.16,
5.20, 5.21\\ 
& \\
\left[\begin{array}{cccc}
uv & 0 & 0 & 0 \\
0 & uv & 0 & 0 \\
0 & 0 & uv & 0 \\
0 & 0 & 0 & uv
\end{array}\right] & 4.3, 4.4, 4.5, 4.9, 5.5, 5.6, 5.9, 5.10, 5.11, 5.12, 5.15, 5.18, 5.19, 5.28, 5.29, 5.30 \\
\end{array}
\]
\[\begin{array}{r|l}
\Phi_X^{M_n,BLE^2}(K) & K \\ \hline
& \\
& \\
\left[\begin{array}{cccc}
4uv & 0 & 0 & 0 \\
0 & 4uv & 0 & 0 \\
0 & 0 & 4uv & 0 \\
0 & 0 & 0 & 4uv
\end{array}\right] & 5.17 \\
& \\
\left[\begin{array}{cccc}
4u^2v & 0 & 0 & 0 \\
0 & 4uv & 0 & 0 \\
0 & 0 & 4u^2v & 0 \\
0 & 0 & 0 & 4u^2v
\end{array}\right] & 5.1, 5.2, 5.17, 5.22, 5.23, 5.24, 5.25.
\end{array}
\]

Comparing the enhancements, we see that 
\begin{itemize}
\item The unenhanced counting invariant divides the knotoids into 
two classes (those with four colorings and those with sixteen),
\item The biqundle counting matrix refines the counting invariant into 
three classes with invariant values
\[
\left[\begin{array}{cccc}
1 & 0 & 0 & 0 \\
0 & 1 & 0 & 0 \\
0 & 0 & 1 & 0 \\
0 & 0 & 0 & 1
\end{array}\right],
\left[\begin{array}{cccc}
0 & 0 & 0 & 0 \\
0 & 4 & 0 & 0 \\
0 & 0 & 0 & 0 \\
0 & 0 & 0 & 0
\end{array}\right] \ \mathrm{and}\
\left[\begin{array}{cccc}
4 & 0 & 0 & 0 \\
0 & 4 & 0 & 0 \\
0 & 0 & 4 & 0 \\
0 & 0 & 0 & 4
\end{array}\right],
\]
\item The biquandle longitude exponent pair polynomial enhancement 
refines the counting invariant into four classes with values 
$uv+3u^2v,$ $4uv$, $16uv$ and $4uv+12u^2v$, and
\item The biquandle longitude exponent pair matrix enhancement
refines the counting invariant into five classes as showin in the table.
\end{itemize}

\end{example}


\section{\large\textbf{Questions}}\label{Q}

We end with some questions for future research. 

A biquandle longitude weight for an $X$-colored knotoid is a kind of 
noncommutative Boltzmann weight with values in the group of permutations
1	of elements of the coloring biquandle $X$. Does a non-commutative analogue
of biquandle homology lurk in the shadows here?

What other enhancements can be defined for the biquandle counting
invariant for knotoids? How do these invariants behave under the connected sum operation on knotoids (that is 
joining knotoids head-to-tail)?

\section*{Acknowledgements}
The first author thanks to Simons Foundation for supporting her to visit to the second author at the Claremont McKenna College, and to the Department of Mathematics at the Claremont McKenna College for the warm hospitality during her stay there.

\bibliography{ng-sn}{}
\bibliographystyle{abbrv}
%
%
%
%
%
%
%
%

\bigskip
\noindent
\textsc{Department of Mathematics \\
National Technical University of Athens \\
Zographou Campus, Iroon Polytechniou 9 \\
Athens 15780, Greece} \\

\noindent
\textsc{Department of Mathematical Sciences \\
Claremont McKenna College \\
850 Columbia Ave. \\
Claremont, CA 91711}

\end{document}